\documentclass[a4paper,10pt]{article}

\usepackage{amsmath,amssymb,amsthm,amscd}
\usepackage[mathscr]{eucal}
\usepackage[all]{xy}
\usepackage{graphicx}

\makeatletter
  
  \@addtoreset{equation}{section}
\makeatother

\newcommand{\mbf}{\mathbf}

\newcommand{\mfrak}{\mathfrak}
\newcommand{\mbb}{\mathbb}
\newcommand{\mrm}{\mathrm}

\newcommand{\vphi}{\varphi}

\newcommand{\aet}{\mathrm{\acute{e}t}}
\newcommand{\cO}{\mathcal{O}}

\newtheorem{theorem}{Theorem}[section]

\newtheorem{lemma}[theorem]{Lemma}
\newtheorem{proposition}[theorem]{Proposition}

\theoremstyle{definition}
\newtheorem{definition}[theorem]{Definition}
\newtheorem{remark}[theorem]{Remark}

\newtheorem*{question}{Question}

\newtheorem*{acknowledgments}{Acknowledgments}

\title{Torsion of abelian varieties and Lubin-Tate extensions}

\author{Yoshiyasu Ozeki\footnote{
Department of Mathematics and Physics, Faculty of Science, Kanagawa University,
  2946 Tsuchiya, Hiratsuka-shi, Kanagawa 259--1293, JAPAN
\endgraf
e-mail: {\tt ozeki@kanagawa-u.ac.jp}}
}

\begin{document}
\maketitle

\begin{abstract}
We show that, for an abelian variety defined over a $p$-adic field $K$
which has potential good reduction,
its torsion subgroup with values in the composite field of $K$ and 
a certain Lubin-Tate extension over a $p$-adic field is finite. 
\end{abstract}


\section{Introduction}

Let $p$ be a prime number and $A$ an abelian variety over a  $p$-adic field $K$
(here, a {\it $p$-adic field} is a finite extension of $\mbb{Q}_p$). 
For an algebraic extension $L/K$, we denote by $A(L)$ the group of 
$L$-rational points of $A$ and also denote by $A(L)_{\mrm{tor}}$ its torsion subgroup.
We are interested in determining whether $A(L)_{\mrm{tor}}$ is finite or not.
The most basic result is given by Mattuck \cite{Ma}; 
$A(L)_{\mrm{tor}}$ is finite if  $L$ is a finite extension of $K$.
Thus our main interest is the case where $L$ is an infinite algebraic extension of $K$.
For this, Imai's result \cite{Im} is well-known.
He showed that $A(K(\mu_{p^{\infty}}))_{\mrm{tor}}$ is finite if  
$A$ has potential good reduction,
where $\mu_{p^{\infty}}$ denotes the group of $p$-power roots of unity
in a fixed separable closure $\overline{K}$ of $K$.
Since the field $K(\mu_{p^{\infty}})$ is the composite field of $K$ and the Lubin-Tate extension 
over $\mbb{Q}_p$ associated with a uniformizer $p$ of $\mbb{Q}_p$, 
we naturally have the following question. 

\begin{question}
Let $A$ be an abelian variety over a $p$-adic field $K$.
Let $k_{\pi}$ be  the Lubin-Tate extension associated with a uniformizer $\pi$ of a $p$-adic field $k$.
Then, is $A(Kk_{\pi})_{\mrm{tor}}$ finite?
\end{question}

\noindent
In the case of Imai's theorem ($k=\mbb{Q}_p$ and $\pi=p$), 
the answer of the question is affirmative for good reduction cases.
However, the question sometimes has a negative answer.
For example, if $A$ is a Tate curve over $K$, $k=\mbb{Q}_p$ and $\pi=p$,
then $A(Kk_{\pi})[p^{\infty}]=A(K(\mu_{p^{\infty}}))[p^{\infty}]$ is clearly infinite.
We also have an example even  for good reduction cases as given  in Remark \ref{normcondi}.

The aim of this paper is to give a sufficient condition 
on $k$ and $\pi$ so that the question has an affirmative answer
for good reduction cases.
Let $k, \pi$ and $k_{\pi}$ be as above.
Let $q$ be the order of the residue field of $k$. 
We denote by $k_G$ the Galois closure of $k/\mbb{Q}_p$.  
We put $d_G=[k_G:\mbb{Q}_p]$ and denote by $e_G$ 
the ramification index of the extension $k_G/k$.
Our main result is as follows
(see  Definitions \ref{Weil1} and \ref{Weil2} for some undefined  notion).

\begin{theorem}
\label{MC}
Let $A$ be an abelian variety over a $p$-adic field $K$ with potential good reduction.
If $\mrm{Nr}_{k/\mbb{Q}_p}(\pi)$
is not a $q$-Weil integer of weight 
$sd_G/t$ for some integers $1\le s\le e_G$ and $1\le t \le sd_G$, 
then $A(Kk_{\pi})_{\mrm{tor}}$ is finite.
\end{theorem}

\noindent
Applying Theorem \ref{MC} to the case where $k=\mbb{Q}_p$ and $\pi=p$,
we can recover Imai's theorem.
We should note that there is an another generalization of Imai's theorem
which is given by Kubo and Taguchi \cite{KT}. 
The main result of {\it loc.\ cit.} states that 
the torsion subgroup of $A(K(K^{1/p^{\infty}}))$ is finite, where
$A$ is an abelian variety over  $K$ with potential good reduction and 
$K(K^{1/p^{\infty}})$ is the extension field of $K$
by adjoining all $p$-power roots of all elements of $K$. 

For the proof of the above theorem,
the essential difficulty appears in the finiteness 
of the $p$-power torsion part $A(Kk_{\pi})[p^{\infty}]$ of $A(Kk_{\pi})_{\mrm{tors}}$.
For this, we proceed our arguments in more general settings. 
We study not only abelian varieties but also \'etale cohomology groups 
of proper smooth varieties.

\begin{theorem}
\label{MTp:var}
Let $X$ be a proper smooth variety over a $p$-adic field $K$
with potential good reduction.
Let $V$ be a $\mrm{Gal}(\overline{K}/K)$-stable subquotient of 
$H^i_{\aet}(X_{\overline{K}},\mbb{Q}_p(r))$ with $i\not=2r$.
Assume that $V^{\mrm{Gal}(\overline{K}/L)}\not=0$ for some finite extension $L/Kk_{\pi}$.
Then  $\mrm{Nr}_{k/\mbb{Q}_p}(\pi)$ 
is a $q$-Weil number of weight $-(i-2r)/h$
for some non-zero $h\in [-i+r,r]\bigcap \left( \bigcup_{1\le s\le e_G} (1/sd_G) \mbb{Z}\right)$.
Moreover, 
$q^r\mrm{Nr}_{k/\mbb{Q}_p}(\pi)^{-h}$ is an algebraic integer.
\end{theorem}

\noindent
Applying Theorem \ref{MTp:var} to the case where $k=\mbb{Q}_p$ and $\pi=p$,
we obtain \cite[Theorem 1.5]{CSW} for $0$-th cohomology groups.
The assumption $i \not=2r$ in Theorem \ref{MTp:var} is essential 
as explained in the Introduction of \cite{KT}. 
The key ingredients for our proof are the theory of locally algebraic representations
(cf.\ \cite{Se})
and some ``weight arguments''
of eigenvalues of Frobenius on various objects.
For weight arguments, we use $p$-adic Hodge theory related with Lubin-Tate characters
and results on weights of a Frobenius operator on crystalline cohomologies (cf.\ \cite{CLS}, \cite{KM}, \cite{Na}).

\vspace{5mm}
\noindent
{\bf Notation :}
In this paper, we fix algebraic closures $\overline{\mbb{Q}}$ and $\overline{\mbb{Q}}_p$
of $\mbb{Q}$ and $\mbb{Q}_p$, respectively, and 
we fix an embedding $\overline{\mbb{Q}}\hookrightarrow \overline{\mbb{Q}}_p$.
If $F$ is a $p$-adic field, 
we denote by  $G_F$ and  $U_F$
the absolute Galois group $\mrm{Gal}(\overline{\mbb{Q}}_p/F)$ of $F$
and the unit group of the integer ring of  $F$,
respectively.
We also denote by 
$F^{\mrm{ur}}$  and $I_F$
the maximal unramified extension of $F$ in $\overline{\mbb{Q}}_p$ and 
the inertia subgroup $\mrm{Gal}(\overline{\mbb{Q}}_p/F^{\mrm{ur}})$ of $G_F$,
respectively.
We set $\Gamma_F:=\mrm{Hom}_{\mbb{Q}_p}(F,\overline{\mbb{Q}}_p)$.
If $F'/F$ is a finite extension,
we denote by $f_{F'/F}$ the residual extension degree of $F/F'$.
that is, the extension degree of the residue fields corresponding to 
$F'/F$. 
We put $f_{F}=f_{F/\mbb{Q}_p}$. 

\begin{acknowledgments}
The author would like to express his sincere gratitude to Professor Yuichiro
Taguchi for giving him useful
advices, especially Remark 2.10.
\end{acknowledgments}

\section{Proofs of main theorems}

Our goal is to prove results in the Introduction.
Let $K,k$ and $E$ be a finite extension of $\mbb{Q}_p$.
Let $\pi$ be a uniformizer of $k$ and 
$k_{\pi}$ the Lubin-Tate extension of $k$ associated with $\pi$.
We denote by $\chi_{\pi}\colon G_k \to k^{\times}$ the Lubin-Tate character
associated with $\pi$.
If we regard $\chi_{\pi}$ as a continuous 
character $k^{\times}\to k^{\times}$ 
by the local Artin map with arithmetic normalization,
then $\chi_{\pi}$ is characterized by the property that $\chi_{\pi}(\pi)=1$
and $\chi_{\pi}(u)=u^{-1}$ for any $u\in U_k$.
\begin{definition}
\label{Weil1}
Let $q_0>1$ be an integer.
A {\it $q_0$-Weil number {\rm (}resp.\ $q_0$-Weil integer{\rm )} of weight $w$} 
is an algebraic number (resp.\ algebraic integer) $\alpha$
such that $|\iota (\alpha)|=q_0^{w/2}$ for all embeddings 
$\iota\colon \mbb{Q(\alpha)}\hookrightarrow \mbb{C}$. 
\end{definition}

\begin{definition}
\label{Weil2}
Let $F$ be a finite extension of $\mbb{Q}_p$ 
with residual extension degree $f=f_F$
and $F_0/\mbb{Q}_p$ the maximal 
unramified subextension of $F/\mbb{Q}_p$.

\noindent
(1)
Let $D$ be a $\vphi$-module over $F$, that is, 
a finite dimensional $F_0$-vector space with 
$\vphi_{F_0}$-semilinear map $\vphi \colon D\to D$.
Then $\vphi^f\colon D\to D$ is a $F_0$-linear map.
We call $\mrm{det}(T-\vphi^f\mid  D)$ 
the {\it characteristic polynomial of $D$}.

\noindent
(2)
For  a $\mbb{Q}_p$-representation $U$ of $G_F$,
we set $D^F_{\mrm{cris}}(U):=(B_{\mrm{cris}}\otimes_{\mbb{Q}_p} U)^{G_F}$
and $D^F_{\mrm{st}}(U):=(B_{\mrm{st}}\otimes_{\mbb{Q}_p} U)^{G_F}$,
which are filtered $\vphi$-modules over $F$. 
Here, $B_{\mrm{cris}}$ and $B_{\mrm{st}}$ are usual $p$-adic period rings.
Note that we  have $D^F_{\mrm{cris}}(U)=D^F_{\mrm{st}}(U)$ if $U$ is crystalline.

\noindent
(3)
 Let $S$ be a set of rational numbers.
Let $U$  be a  potentially semi-stable $\mbb{Q}_p$-representation of $G_F$.
Suppose that $U|_{G_{F'}}$ is semi-stable for a finite extension $F'$ of $F$ 
with residue field $\mbb{F}_{q'}$.  
We say that $U$ has {\it Weil weights in $S$} if 
any root of the characteristic polynomial of $D^{F'}_{\mrm{st}}(U)$
is a $q'$-Weil number of weight $w$ for some $w\in S$.
(Note that this definition does not depend on the choice of $F'$.)
\end{definition}

The following theorem is a key to the proof of our main results.

\begin{theorem}
\label{MTp}
Let $S$ be a subset of $\mbb{Q} \smallsetminus  \{0 \}$.
Let $V$ be a 
semi-stable $\mbb{Q}_p$-representation of $G_K$ with 
Hodge-Tate weights in $[h_1,h_2]$.
Assume that $V$ has Weil weights in $S$ and 
$V^{\mrm{Gal}(\overline{K}/L)}\not=0$ for some finite extension $L/Kk_{\pi}$. Then

\noindent
{\rm (1)}  $\mrm{Nr}_{k/\mbb{Q}_p}(\pi)$ 
is a $q$-Weil number of weight $-w/h$
for some $w\in S$ and some non-zero
$h\in [h_1,h_2]\bigcap \left( \bigcup_{1\le s\le e_G} (1/sd_G) \mbb{Z}\right)$.

\noindent
{\rm (2)} If the coefficients of 
the characteristic polynomial of $D^K_{\mrm{st}}(V(-r))$
are algebraic integers for some integer $r$, then 
we can choose $h$ in (1) so that 
$q^r\mrm{Nr}_{k/\mbb{Q}_p}(\pi)^{-h}$ is an algebraic integer.
\end{theorem}

\subsection{Proof of Theorem \ref{MTp}}

In this section, we prove Theorem \ref{MTp}. We begin with some lemmas.

\begin{lemma}
\label{lem0}
Let $F$ be a $p$-adic field and 
$(n_{\sigma})_{\sigma\in \Gamma_F}$  a family of integers.
If there exists an open subgroup $U$ of $U_F$ with the property that
$\prod_{\sigma\in \Gamma_F} \sigma(x)^{n_{\sigma}}=1$ for any $x\in U$,
then we have $n_{\sigma}=0$ for any $\sigma\in \Gamma_F$. 
\end{lemma}

\begin{proof}
Replacing $U$ by a finite index subgroup,
we may assume that the $p$-adic logarithm map is defined on $U$.
Then we have $\sum_{\sigma\in \Gamma_F} n_{\sigma}\sigma(\log x)=0$ for any $x\in U$.
Since $\log U$ is an open ideal of the ring of integers of $F$,
we obtain
 $\sum_{\sigma\in \Gamma_F} n_{\sigma} \sigma (y)=0$ for any $y\in F$.
Although the desired fact $n_{\sigma}=0$ for any $\sigma\in \Gamma_F$
follows from  Dedekind's theorem \cite[\S 6, no. 2, Corollaire 2]{Bo} immediately, 
we also give a direct proof for this.
Take any $\alpha\in F$ such that $F=\mbb{Q}_p(\alpha)$
and let $\Gamma_F=\{\sigma_1=\mrm{id}, \sigma_{2},\dots ,\sigma_c \}$
where $c:=[F:\mbb{Q}_p]$.
Then we have 
$(n_{\sigma_1}, n_{\sigma_2},\dots ,n_{\sigma_c})X=\mbf{0}$
where $X$ is the $c\times c$ matrix with $(i,j)$-th component
$\sigma_i(\alpha)^{j-1}$.
Since $\mrm{det}\ X=\prod_{j>i} (\sigma_j(\alpha)-\sigma_i(\alpha))\not =0$,
we obtain $n_{\sigma_1}=n_{\sigma_2}=\cdots =n_{\sigma_c}=0$.
\end{proof}

\begin{lemma}
\label{lem1}
Let $E$ be a $p$-adic field and
$V$ an $E$-representation of $G_K$.
Assume that $k/\mbb{Q}_p$ is Galois,
$V$ is Hodge-Tate and 
 the $G_{Kk_{\pi}}$-action on $V$ factors through a finite quotient.
Then, there exist finite extensions $K'/K$ and $E'/E$ with $K',E'\supset k$   
such that any Jordan-H\"ollder factor of $(V\otimes_E E')|_{G_{K'}}$
is of the form 
$E'(\prod_{\sigma \in \Gamma_k}\sigma^{-1}\circ \chi^{r_{\sigma}}_{\pi})$
for some $r_{\sigma}\in \mbb{Z}$. 
Moreover, $r_{\sigma}$ is a Hodge-Tate weight of $V$.
\end{lemma}

\begin{proof}
Replacing $K$ by a finite extension,
we may assume that 
$G_{Kk_{\pi}}$ acts on $V$ trivial
and $K$ is a finite Galois extension of $k$.
Since the $G_K$-action on $V$ factors through the abelian group $\mrm{Gal}(Kk_{\pi}/K)$,
it follows from Schur's lemma that, for a finite extension $E'/E$ large enough, 
any  Jordan-H\"ollder factor $W$ of $V\otimes_E E'$ is of dimension $1$.  
Our goal is to show that $W$ is of   the required form. 
We may assume  $E'=E\supset K$.

Let $\rho\colon G_K\to GL_E(W)\simeq E^{\times}$
be the continuous homomorphism given by the $G_K$-action on $W$. 
Let $\tilde{E}$ be the Galois closure of $E/\mbb{Q}_p$ and take any
finite extension $K''/K$ which contains $\tilde{E}$.
Since $W$ is Hodge-Tate,
it follows from \cite[Chapter III,\ A.\ 5, Theorem 2]{Se} that
there exists an open subgroup $I$ of $I_{K''}$
such that $\rho=\prod_{\sigma\in \Gamma_{E}} \sigma^{-1}\circ \chi_{\sigma E}^{n_{\sigma}}$
on $I$ for some integer $n_{\sigma}$. 
Here, $\chi_{\sigma E}\colon G_{\sigma E}\to U_{\sigma E}$ is the
Lubin-Tate character associated  with $\sigma E$ 
(it depends on the choice of a uniformizer of $\sigma E$,
but its restriction to the inertia subgroup does not).
Put $\tilde{\rho}=\prod_{\sigma\in \Gamma_{E}} \sigma^{-1}\circ \chi_{\sigma E}^{n_{\sigma}}$, 
considered as a character of $G_{K''}$. 
Replacing $K''$ by a finite extension,
we may assume the following:
\begin{itemize}
\item[--] $K''/\mbb{Q}_p$ is Galois, 
$\mrm{Gal}(k_{\pi}/(k_{\pi}\cap K''))$ is torsion free and $\rho=\tilde{\rho}$ on $I_{K''}$.
\end{itemize} 
Since  $\rho|_{G_{Kk_{\pi}}}$ is trivial,
we have that $\tilde{\rho}$ is trivial on $I_{K''}\cap G_{Kk_{\pi}}
=G_{(K'')^{\mrm{ur}}k_{\pi}}$. 
Hence, putting  $N'=\mrm{Gal}((K'')^{\mrm{ur}}k_{\pi}/(K'')^{\mrm{ur}})$,
 we may regard $\tilde{\rho}|_{I_{K''}}$
as a representation of $N'$. 
Put $N=\mrm{Gal}(k^{\mrm{ur}}k_{\pi}/k^{\mrm{ur}})$.
Then $N'$ is canonically isomorphic to 
a torsion free finite index subgroup of $N\simeq U_k$, and thus 
we regard $N'$ as a  subgroup of $N$.

Now we claim that $\tilde{\rho}\colon N'\to \tilde{E}^{\times}$ extends to  
a continuous character $\hat{\rho}\colon N\to \overline{\mbb{Q}}_p^{\times}$. 
It follows from the theory of elementary divisors that 
we may regard 
$N=N_{\mrm{tor}}\oplus (\oplus^d_{i=1} \mbb{Z}_p)\supset 
\{ 0 \}\oplus (\oplus^d_{i=1} p^{m_i}\mbb{Z}_p)=N'$
with some integer $m_i\ge 0$. Here, $N_{\mrm{tor}}$ is the torsion subgroup of $N$
and $d:=[k:\mbb{Q}_p]$.  
Hence it suffices to show that 
any continuous character $p^m\mbb{Z}_p\to \overline{\mbb{Q}}_p^{\times}$ with $m>0$
extends to $\mbb{Z}_p\to \overline{\mbb{Q}}_p^{\times}$, but this is clear.

By local class field theory,
we may regard $\tilde{\rho}|_{I_{K''}}$ and $\hat{\rho}$
as characters of  $U_{K''}$ and $U_k$, respectively.
It follows from the construction of $\hat{\rho}$
that
we have $\tilde{\rho}(x)=\hat{\rho}(\mrm{Nr}_{K''/k}(x))$ for $x\in U_{K''}$.
In particular, we have 
\begin{equation}
\label{eq:1}
\tilde{\rho}(x)=\tilde{\rho}(\tau x)\ 
\end{equation}
for  $x\in U_{K''}$ and $\tau\in \mrm{Gal}(K''/k)$. 
On the other hand, by definition of $\tilde{\rho}$ 
and the condition that $K''/\mbb{Q}_p$ is Galois,
we have 
\begin{equation}
\label{eq:2}
\tilde{\rho}(x)  = \prod_{\sigma\in \Gamma_{E}} 
\sigma^{-1}\mrm{Nr}_{K''/\sigma E}(x^{-1})^{n_{\sigma}}
=\prod_{\tilde{\sigma}\in \Gamma_{K''}} 
\tilde{\sigma}^{-1}(x^{-1})^{n_{\tilde{\sigma}}}
\end{equation}
for  $x\in U_{K''}$ where $n_{\tilde{\sigma}}:=n_{\sigma}$ if $\tilde{\sigma}|_E=\sigma$.

We claim that $n_{\tilde{\sigma}}=n_{\tilde{\sigma}'}$ if 
$\tilde{\sigma}|_k=\tilde{\sigma}'|_k$.
By \eqref{eq:1} and \eqref{eq:2},
we have 
\begin{equation}
\label{eq:3}
\prod_{\tilde{\sigma}\in \Gamma_{K''}} 
\tilde{\sigma}^{-1}(x^{-1})^{n_{\tau\tilde{\sigma}}}
=\prod_{\tilde{\sigma}\in \Gamma_{K''}} 
\tilde{\sigma}^{-1}(x^{-1})^{n_{\tilde{\sigma}}}
\end{equation}
for  $x\in U_{K''}$ and $\tau\in \mrm{Gal}(K''/k)$. 
Choosing a lift $\hat{\sigma}\in \Gamma_{K''}$ for each element 
of $\mrm{Gal}(k/\mbb{Q}_p)$,
we have a decomposition 
$\Gamma_{K''}= \bigcup_{\hat{\sigma}} \hat{\sigma} \mrm{Gal}(K''/k)$.
Since $k/\mbb{Q}_p$ is Galois, 
we see that  $\mrm{Gal}(K''/k)$ acts on 
$\hat{\sigma} \mrm{Gal}(K''/k)$ stable and this action is 
transitive. 
By Lemma \ref{lem0}, we know that 
the family $(n_{\tilde{\sigma}})_{\tilde{\sigma}\in \Gamma_{K''}}$
is determined uniquely by the restriction of 
$\prod_{\sigma\in \Gamma_{K''}} 
(\tilde{\sigma}^{-1})^{n_{\tilde{\sigma}}}$
to any open subgroup of $U_{K''}$.
Hence the equation \eqref{eq:3} gives  $n_{\tilde{\sigma}}=n_{\tilde{\sigma}'}$ if
$\tilde{\sigma}|_k=\tilde{\sigma}'|_k$ as desired.

For any $\sigma\in \Gamma_k$, we define 
$r_{\sigma}:=n_{\tilde{\sigma}}$ for a lift $\tilde{\sigma}\in \Gamma_{K''}$
 of $\sigma$, which is independent of the choice of $\tilde{\sigma}$ 
by the claim just above.
Then we see 
$
\tilde{\rho}(x)  
=\prod_{\tilde{\sigma}\in \Gamma_{K''}} 
\tilde{\sigma}^{-1}(x^{-1})^{n_{\tilde{\sigma}}}
= \prod_{\sigma\in \Gamma_k} 
\sigma^{-1}\mrm{Nr}_{K''/k}(x^{-1})^{r_{\sigma}}
$
for  $x\in U_{K''}$. This implies 
$$
\tilde{\rho}=
\prod_{\sigma\in \Gamma_k} 
\sigma^{-1}\circ \chi_{\pi}^{r_{\sigma}}
$$
on $I_{K''}$. 
Now we define  $\psi \colon G_K\to E^{\times}$
by $\psi:= \rho\cdot \left(\prod_{\sigma\in \Gamma_k} 
\sigma^{-1}\circ \chi_{\pi}^{r_{\sigma}}\right)^{-1}$.
Then $\psi$ is trivial on $I_{K''}$ since $\rho=\tilde{\rho}=\prod_{\sigma\in \Gamma_k} 
\sigma^{-1}\circ \chi_{\pi}^{r_{\sigma}}$ on $I_{K''}$.
Furthermore,  $\psi$ is trivial on $G_{Kk_{\pi}}$ since 
$\chi_{\pi}$ and $\rho$ is trivial on $G_{Kk_{\pi}}$.
Therefore, putting $K'=(K'')^{\mrm{ur}}\cap Kk_{\pi}$,
then $K'/K$ is a finite extension and $\psi$ is trivial on $G_{K'}$.

Finally, we note that $r_{\sigma}$ is a Hodge-Tate weight of $V$
by \cite[Chapter III,\ A.5,\ Theorem 2]{Se}.
This is the end of the proof.
\end{proof}

\begin{lemma}
\label{lem1'}
Let $E$ be a $p$-adic field and
$V$ an $E$-representation of $G_K$.
Assume that $k/\mbb{Q}_p$ is Galois,
$V$ is potentially semi-stable with 
Hodge-Tate weights in $[h_1,h_2]$ and 
the $G_{Kk_{\pi}}$-action on $V$ factors through a finite quotient.
Then, there exists a finite extension $K'/Kk$  
which satisfy the following property:
$V|_{G_{K'}}$ is semi-stable and, 
for any root $\alpha$ of  the characteristic polynomial 
of $D^{K'}_{\mrm{st}}(V)$,
we have 
$$
\alpha=a^{f_{K'/k}},\quad 
a=\prod_{\tau\in \Gamma_k} \tau(\pi)^{-n_{\tau}} 
$$
for some integers $(n_{\tau})_{\tau\in \Gamma_k}$ such that
 $dh_1\le \sum_{\tau\in \Gamma_k} n_{\tau} \le dh_2$.
Here,  $d:=[k:\mbb{Q}_p]$.
\end{lemma}

\begin{proof}
By Lemma \ref{lem1}, there exist  finite extensions $K'/K$ and $E'/E$ 
with $E',K'\supset k$
which satisfy the following: 
\begin{itemize}
\item[--] $V|_{G_{K'}}$ is semi-stable and any Jordan-H\"ollder factor $W$ of $(V\otimes_{E} E')|_{G_{K'}}$
is of the form 
$E'(\prod_{\sigma \in \Gamma_k}\sigma^{-1}\circ \chi^{r_{\sigma}}_{\pi})$
for some $r_{\sigma}\in [h_1,h_2]$.
In particular, $W$ is crystalline.
\end{itemize}
Replacing $E$ by a finite extension, we may assume $E'=E$.
Now we take a root $\alpha$ of  the characteristic polynomial 
of $D^{K'}_{\mrm{st}}(V)$,
and choose $W$ so that $\alpha$ is a root of the  
 the characteristic polynomial 
of $D^{K'}_{\mrm{cris}}(W)$.

To study $\alpha$,  we  first consider the characteristic polynomial 
of $D^{K'}_{\mrm{cris}}(E(\sigma^{-1}\circ \chi^{r_{\sigma}}_{\pi}))$
for $\sigma\in \Gamma_k$. 
Let $K_0'$ be the maximal unramified subextension of $K'/\mbb{Q}_p$
and put  $q'=p^{f_{K'}}$. 
We note that we have an isomorphism 
$k(\sigma^{-1}\circ \chi^{r_{\sigma}}_{\pi})^{\mrm{ss}}
\simeq k(\chi^{r_{\sigma}}_{\pi})^{\mrm{ss}}$
of $\mbb{Q}_p[G_{K'}]$-modules (here, ``$\mrm{ss}$'' stands for the semi-simplification of $\mbb{Q}_p[G_{K'}]$-modules).
In fact, for any $g\in G_{K'}$, we have
\begin{align*}
\mrm{Tr}_{\mbb{Q}_p}(g \mid k(\sigma^{-1}\circ \chi^{r_{\sigma}}_{\pi}))
& = \mrm{Tr}_{k/\mbb{Q}_p}
(\mrm{Tr}_{k}(g \mid k(\sigma^{-1}\circ \chi^{r_{\sigma}}_{\pi})) 
= \mrm{Tr}_{k/\mbb{Q}_p}(\sigma^{-1}\chi^{r_{\sigma}}_{\pi}(g)) \\
& = \mrm{Tr}_{k/\mbb{Q}_p}(\chi^{r_{\sigma}}_{\pi}(g)) 
 = \mrm{Tr}_{k/\mbb{Q}_p}
(\mrm{Tr}_{k}(g \mid k(\chi^{r_{\sigma}}_{\pi})) 
= \mrm{Tr}_{\mbb{Q}_p}(g \mid k(\chi^{r_{\sigma}}_{\pi})).
\end{align*}
(Here, for a representation $U$ of a group $G$ over a field $F$ and $g\in G$,
we denote by $\mrm{Tr}_F(g \mid U)$ the trace of the $g$-action on the  
$F$-vector space $U$.) 
Therefore, we have 
\begin{equation}
\label{char2}
\mrm{det}(T-\vphi^{f_{K'}} \mid D^{K'}_{\mrm{cris}}(E(\sigma^{-1}\circ \chi^{r_{\sigma}}_{\pi})))
=\mrm{det}(T-\vphi^{f_{K'}} \mid D^{K'}_{\mrm{cris}}(k(\chi^{r_{\sigma}}_{\pi})))^{[E:k]}.
\end{equation}
To study the roots of \eqref{char2}, we recall the explicit description of 
$D^k_{\mrm{cris}}(k(\chi^{-1}_{\pi}))$ (cf.\ \cite[Proposition B.4]{Con}. 
See also \cite[Proposition 9.10]{Col}). 
Let $k_0$ be the maximal unramified subextension of $k/\mbb{Q}_p$. 
By definition, we have $f_k=[k_0:\mbb{Q}_p]$ and $q=p^{f_k}$.
Then $D^k_{\mrm{cris}}(k(\chi^{-1}_{\pi}))$ is a free $(k_0 \otimes_{\mbb{Q}_p} k)$-module 
of rank one, and we can take a basis $\mbf{e}$ of  $D^k_{\mrm{cris}}(k(\chi^{-1}_{\pi}))$
such that $\vphi^{f_k}(\mbf{e})=(1\otimes \pi)\mbf{e}$. 
We claim 
\begin{equation}
\label{char}
\mrm{det}(T-\vphi^{f_k} \mid D^k_{\mrm{cris}}(k(\chi^{-1}_{\pi})))=\prod_{0\le i\le f_k-1}E^{\vphi^i}(T)
\end{equation}
where $E(T)=T^e+\sum^{e-1}_{j=0}a_jT^j \in k_0[T]$ is the minimal polynomial 
of $\pi$ over $k_0$ and $E^{\vphi^i}(T)=T^e+\sum^{e-1}_{j=0}\vphi^i(a_j)T^j$. 
To show this, it suffices to show that 
the characteristic polynomial of the homomorphism
$1\otimes \pi\colon k_0\otimes_{\mbb{Q}_p} k \to k_0\otimes_{\mbb{Q}_p} k $
of  $k_0$-modules
coincides with the right hand side of \eqref{char}. 
(Here, the $k_0$-action on  $k_0\otimes_{\mbb{Q}_p} k$
is given by $a.(x\otimes y):=ax\otimes y$ for $a,x\in k_0$ and $y\in k$.)
We consider a natural isomorphism 
$$
k_0\otimes_{\mbb{Q}_p} k_0\simeq \oplus_{j\in \mbb{Z}/f_k\mbb{Z}} k_{0,j},\ a\otimes b\mapsto (a\vphi^j(b))_{j} 
$$
where $k_{0,j}=k_0$. For $0\le s\le f_k-1$, let $e_s\in k_0\otimes_{\mbb{Q}_p} k_0$ be the element
which corresponds to $(\delta_{sj})_j\in \oplus_{j\in \mbb{Z}/f_k\mbb{Z}} k_{0,j}$
where $\delta_{sj}$ is the Kronecker delta.  
Then $\{e_j(1\otimes \pi^i) \mid 0\le j\le f_k-1,\ 0\le i\le e-1\}$
is a $k_0$-basis of $k_0\otimes_{\mbb{Q}_p} k$.
We see that 
the matrix of   $1\otimes \pi\colon k_0\otimes_{\mbb{Q}_p} k \to k_0\otimes_{\mbb{Q}_p} k $
associated with the ordered basis
$
\langle e_0,\dots ,e_{f_k-1},e_0(1\otimes \pi),\dots ,e_{f_k-1}(1\otimes \pi),
\dots ,e_0(1\otimes \pi^{e-1}),\dots ,e_{f_k-1}(1\otimes \pi^{e-1}) \rangle
$
is
$$
\left(
\begin{array}{cccc}
O  & O & \cdots & -A_0\\
I_{f_k}  & O & \cdots  & -A_1 \\
\vdots    & \ddots    &    & \vdots \\
O    & \cdots    &    I_{f_k}  & -A_{e-1}
\end{array}
\right) 
$$
where $I_{f_k}$ is the $f_k\times f_k$ identity matrix and  
$A_i$ is the $f_k\times f_k$ diagonal matrix
with diagonal entries $a_i,\vphi(a_i),\dots ,\vphi^{f_k-1}(a_i)$.
Now it is an easy exercise to check that the characteristic polynomial of this matrix 
is $\prod_{0\le i\le f_k-1}E^{\vphi^i}(T)$ as desired.

By the claim \eqref{char}, we know that any root of the characteristic polynomial of
$D^{K'}_{\mrm{cris}}(k(\chi_{\pi}))$ is of the form $\tau(\pi)^{-f_{K'/k}}$
for some $\tau\in \Gamma_k$.
Hence, by \eqref{char2}, 
any root of the characteristic polynomial of
$D^{K'}_{\mrm{cris}}(E(\sigma^{-1}\circ \chi^{r_{\sigma}}_{\pi}))$ 
is of the form $\prod_{\tau\in \Gamma_k} \tau(\pi)^{-f_{K'/k}n^{\sigma}_{\tau}}$
with $\sum_{\tau\in \Gamma_k} n^{\sigma}_{\tau} =r_{\sigma}$.
Therefore,
since $\alpha$ is a root of the characteristic polynomial of
$D^{K'}_{\mrm{cris}}(W)=
D^{K'}_{\mrm{cris}}(E(\prod_{\sigma\in \Gamma_k} \sigma^{-1}\circ \chi^{r_{\sigma}}_{\pi}))$, 
we have
$$
\alpha=\prod_{\tau\in \Gamma_k} \tau(\pi)^{-f_{K'/k}n_{\tau}} 
$$
with $-\sum_{\tau\in \Gamma_k} n_{\tau} 
=-\sum_{\sigma\in \Gamma_k}r_{\sigma}=:R$.
We note that 
$R$ is an integer such that $-dh_2\le R\le -dh_1$
since we have $r_{\sigma}\in [h_1,h_2]$. 
This completes the proof.
\end{proof}

We need the following two standard lemmas which describe inclusion properties of 
two Lubin-Tate extensions.

\begin{lemma}
\label{lemma:LText}
Let $k_2/k_1$ be a finite extension of $p$-adic fields
with residual extension degree $f$.
For $i=1,2$, let $\pi_i$ be a uniformizer of $k_i$ 
and $k_{i,\pi_i}/k_i$
the Lubin-Tate extension associated with $\pi_i$.

\noindent
{\rm (1)} We have $\mrm{Nr}_{k_2/k_1}(\pi_2)=\pi_1^{f}$
if and only if  $k_{1,\pi_1}\subset k_{2,\pi_2}$.

\noindent
{\rm (2)}  $\pi_1^{-f} \mrm{Nr}_{k_2/k_1}(\pi_2)$ is a root of unity
if and only if  
there exists a finite extension $M/k_{2,\pi_2}$ 
such that
$k_{1,\pi_1}\subset M$. 
If this is the case, we can take $M$ to be 
the degree $\sharp \mu_{\infty}(k_1)$ subextension in $k^{\mrm{ab}}_2/k_{2,\pi_2}$. 
Here, $\mu_{\infty}(k_1)$ is the set of roots of unity in $k_1$,
\end{lemma}

\begin{proof}
For $i=1,2$, we  denote by 
$k_i^{\mrm{ur}}$ and $k_i^{\mrm{ab}}$ the maximal unramified extension of $k_i$
and the maximal abelian extension of $k_i$, respectively.
We recall that 
the Artin map $\mrm{Art}_{k_i}\colon k_i^{\times}\to  \mrm{Gal}(k_i^{\mrm{ab}}/k_i)$ 
associated with $k_i$
satisfies 
$\mrm{Art}_{k_i}(\pi_i)|_{k_{i,\pi_i}}=\mrm{id}$ and 
$\mrm{Art}_{k_i}(\pi_i)|_{k_i^{\mrm{ur}}}=\mrm{Frob}_{k_i}$,
where $\mrm{Frob}_{k_i}$ is the geometric Frobenius of $k_i$.

\noindent
(1) 
Suppose  $\mrm{Nr}_{k_2/k_1}(\pi_2)=\pi_1^{f}$.
For any lift $\sigma\in G_{k_2}$ of $\mrm{Art}_{k_2}(\pi_2)$, we have 
$$
\sigma|_{k_{1,\pi_1}}
=(\mrm{Art}_{k_2}(\pi_2)|_{k_1^{\mrm{ab}}})|_{k_{1,\pi_1}}
=\mrm{Art}_{k_1}(\mrm{Nr}_{k_2/k_1}(\pi_2))|_{k_{1,\pi_1}}
=\mrm{Art}_{k_1}(\pi_1)^f|_{k_{1,\pi_1}}
=\mrm{id}
$$
Since the intersection of the fixed fields (in $\overline{\mbb{Q}}_p$)
of such $\sigma$'s is $k_{2,\pi_2}$,
we obtain the desired result.

Conversely, suppose $k_{1,\pi_1}\subset k_{2,\pi_2}$. Then we have
$$
\mrm{Art}_{k_1}(\mrm{Nr}_{k_2/k_1}(\pi_2))|_{k_{1,\pi_1}}
=\mrm{Art}_{k_2}(\pi_2)|_{k_{1,\pi_1}}
=(\mrm{Art}_{k_2}(\pi_2)|_{k_{2,\pi_2}})|_{k_{1,\pi_1}}
=\mrm{id}
$$
and
$$
\mrm{Art}_{k_1}(\mrm{Nr}_{k_2/k_1}(\pi_2))|_{k_1^{\mrm{ur}}}
=\mrm{Art}_{k_2}(\pi_2)|_{k_1^{\mrm{ur}}}
=(\mrm{Art}_{k_2}(\pi_2)|_{k_2^{\mrm{ur}}})|_{k_1^{\mrm{ur}}}
=\mrm{Frob}_{k_2}|_{k_1^{\mrm{ur}}}
=\mrm{Frob}^f_{k_1}.
$$
Thus we have 
$\mrm{Art}_{k_1}(\mrm{Nr}_{k_2/k_1}(\pi_2))
=\mrm{Art}_{k_1}(\pi_1^f)$, which shows  $\mrm{Nr}_{k_2/k_1}(\pi_2)=\pi_1^{f}$.

\noindent
(2) A very similar proof to that of (1) proceeds.
Suppose that $\pi_1^{-f} \mrm{Nr}_{k_2/k_1}(\pi_2)$ is a root of unity.
If we denote by $h$ the order of the set of roots of unity in $k_1$,
then we have $\mrm{Nr}_{k_2/k_1}(\pi_2^h)=\pi_1^{fh}$.
We see that 
any lift $\sigma\in G_{k_2}$ of $\mrm{Art}_{k_2}(\pi_2^h)$ fixes $k_{1,\pi_1}$.
This implies that $k_{1,\pi_1}$ is contained in the degree $h$ subextension in 
$k^{\mrm{ab}}_2/k_{2,\pi_2}$.

Suppose that there exists a finite extension $M/k_{2,\pi_2}$ 
such that $k_{1,\pi_1}\subset M$. Then $M':=k_{1,\pi_1}k_{2,\pi_2}$ is 
a finite subextension in  $k^{\mrm{ab}}_2/k_{2,\pi_2}$.
Put $h=[M':k_{2,\pi_2}]$.
Since $\mrm{Art}_{k_2}(\pi_2^h)|_{M'}$ is the identity map,
we have
$
\mrm{Art}_{k_1}(\mrm{Nr}_{k_2/k_1}(\pi_2^h))|_{k_{1,\pi_1}}
=\mrm{id}
$
and
$
\mrm{Art}_{k_1}(\mrm{Nr}_{k_2/k_1}(\pi_2^h))|_{k_1^{\mrm{ur}}}
=\mrm{Frob}^{fh}_{k_1}.
$
Thus we have 
$\mrm{Art}_{k_1}(\mrm{Nr}_{k_2/k_1}(\pi_2^h))
=\mrm{Art}_{k_1}(\pi_1^{fh})$, which shows  
$\mrm{Nr}_{k_2/k_1}(\pi_2^h)=\pi_1^{fh}$.

\end{proof}

We recall that  $k_G$ is the Galois closure of 
$k/\mbb{Q}_p$ and $d_G:=[k_G:\mbb{Q}_p]$.

\begin{lemma}
\label{lemma:gal}
There exist a finite unramified extension $k'/k_G$ 
and a uniformizer $\pi'$ of $k'$ 
which satisfy the following.
\begin{itemize}
\item $\mrm{Nr}_{k'/k}(\pi')=\pi^{f_{k'/k}}$,
\item $k_{\pi}\subset k'_{\pi'}$, where $k'_{\pi'}$ is the Lubin-Tate extension of $k'$ 
associated with $\pi'$,
\item the extension $k'/\mbb{Q}_p$ is Galois, and 
\item $[k':\mbb{Q}_p]=sd_G$ for some integer $1\le s\le e_G$.
\end{itemize}
\end{lemma}

\begin{proof}
Let $k_{G,0}/k$ be the maximal unramified subextension in $k_G/k$.
By \cite[Chapter V, \S 6, Proposition 10]{Se},
there exists an unramified extension $\tilde{k}_0$ over $k_{G,0}$
of degree at most $[k_G:k_{G,0}](=e_G)$ such that
$\pi=\mrm{Nr}_{k'/\tilde{k}_0}(\pi')$ 
for some $\pi'\in (k')^{\times}$, where 
$k':=k_G\tilde{k}_0$.
Since $k_G/\mbb{Q}_p$ is Galois and 
$k'/k_G$ is unramified, 
we see that $k'/\mbb{Q}_p$ is Galois.
We also see that $\pi'$ is a uniformizer of $k'$.
Since $k_G\cap \tilde{k}_0=k_{G,0}$,
we have $[k':k_G]=[\tilde{k}_0:k_{G,0}]\le e_G$.
Thus we obtain $[k':\mbb{Q}_p]=[k':k_G][k_G:\mbb{Q}_p]=sd_G$ 
for some integer $1\le s\le e_G$.
Furthermore, we have 
$\mrm{Nr}_{k'/k}(\pi')
=\mrm{Nr}_{\tilde{k}_0/k}(\mrm{Nr}_{k'/\tilde{k}_0}(\pi'))
=\mrm{Nr}_{\tilde{k}_0/k}(\pi)
=\pi^{f_{k'/k}}$. 
By Lemma \ref{lemma:LText},
we have $k_{\pi}\subset k'_{\pi'}$.
\end{proof}

Now we are ready to prove Theorem \ref{MTp}.

\begin{proof}[Proof of Theorem \ref{MTp}]
First we consider the case where $k/\mbb{Q}_p$ is Galois.
Replacing $L$ by a finite extension,
we may assume that 
$L/K$ is Galois. Then $V^{G_L}$ is a $G_K$-stable submodule of $V$.
By Lemma \ref{lem1'},
there exists a finite extension  $K'/Kk$
such that 
any root $\alpha$ 
of the characteristic polynomial of
$D^{K'}_{\mrm{st}}(V^{G_L})$ 
is of the form 
$$
\alpha=a^{f_{K'/k}},\quad 
a=\prod_{\tau\in \Gamma_k} \tau(\pi)^{-n_{\tau}} 
$$
with some integers $(n_{\tau})_{\tau\in \Gamma_k}$
such that $dh_1\le \sum_{\tau\in \Gamma_k} n_{\tau}\le dh_2$.
Here, $d:=[k:\mbb{Q}_p]$.
Put $R:=-\sum_{\tau\in \Gamma_k} n_{\tau}$.
Then we have 
\begin{equation}
\label{norm}
\prod_{\sigma\in \Gamma_k} \sigma(a)=
\prod_{\tau\in \Gamma_k} \prod_{\sigma\in \Gamma_k} \sigma \tau(\pi)^{-n_{\tau}}
=\prod_{\tau\in \Gamma_k} \mrm{Nr}_{k/\mbb{Q}_p}(\pi)^{-n_{\tau}}
=\mrm{Nr}_{k/\mbb{Q}_p}(\pi)^R.
\end{equation}
Since $V^{G_L}$ has Weil weights in $S$,
we see that $\sigma(a)$ is a $q$-Weil number of weight $w\in S$ 
for any $\sigma\in \Gamma_k$.
Thus it follows from the condition $w\not=0$ and the equation \eqref{norm}
that we have $R\not=0$.
Therefore, we obtain that 
$\mrm{Nr}_{k/\mbb{Q}_p}(\pi)$ is a $q$-Weil number of weight $-w/h$
where $h:=-R/d\in [h_1,h_2]\cap (1/d)\mbb{Z}$. 
This shows Theorem \ref{MTp} (1).
Now Theorem \ref{MTp} (2) follows from the fact that we have 
$(q^r\mrm{Nr}_{k/\mbb{Q}_p}(\pi)^{-h})^d=\mrm{Nr}_{k/\mbb{Q}_p}(q^ra)$
and $(q^ra)^{f_{K'/k}}=q_{K'}^r\alpha$ is a root of 
the characteristic polynomial of $D^K_{\mrm{st}}(V(-r))$.
Thus we obtained a  proof of Theorem \ref{MTp} in the case where $k/\mbb{Q}_p$ is Galois.

Next we consider the case where $k/\mbb{Q}_p$ is not necessarily Galois. 
Take a finite extension $k'/k_G$ and a uniformizer $\pi'$ of $k'$ 
as in Lemma \ref{lemma:gal}.
Put $d'=[k':\mbb{Q}_p]$. 
We have $d'=sd_G$ for some $1\le s\le e_G$.
Let $q'$ be the order of the residue field of $k'$.
Let $L'$ be the composite field of $L$ and $k'_{\pi'}$, 
which is a finite extension of $Kk'_{\pi'}$.
Assume that $V^{G_L}$ is not zero.
Since $V^{G_{L'}}$ is also not zero 
and the extension $k'/\mbb{Q}_p$ is Galois,
we know that 
$\mrm{Nr}_{k'/\mbb{Q}_p}(\pi')$
is a $q'$-Weil number of weight  $-w/h$
for some $w\in S$ and $h \in [h_1,h_2]\cap (1/d')\mbb{Z}$. 
By the equation $\mrm{Nr}_{k'/k}(\pi')=\pi^{f_{k'/k}}$, 
we have $\mrm{Nr}_{k'/\mbb{Q}_p}(\pi')
=(\mrm{Nr}_{k/\mbb{Q}_p}(\pi))^{f_{k'/k}}$,
and hence
$\mrm{Nr}_{k/\mbb{Q}_p}(\pi)$
is a $q$-Weil number of weight 
$-w/h$.
Furthermore, we have 
$q'^{r}\mrm{Nr}_{k'/\mbb{Q}_p}(\pi')^{-h}
=(q^{r}\mrm{Nr}_{k/\mbb{Q}_p}(\pi)^{-h})^{f_{k'/k}}$.
This completes the proof of 
Theorem \ref{MTp}.
\end{proof}

\subsection{Proofs of Theorems \ref{MC} and  \ref{MTp:var}}
We  prove Theorems \ref{MC} and  \ref{MTp:var} in the Introduction.
We start with a proof of Theorem \ref{MTp:var}.

\begin{proof}[Proof of Theorem \ref{MTp:var}]
Let the notation be as in the theorem. 
Replacing $K$ by a finite extension,
we may assume that $X$ has good reduction over $K$.
Then  we know that $V$ is crystalline
with Hodge-Tate weights in $[-i+r,r]$.
Let $K_0$ be the maximal unramified subextension of $K/\mbb{Q}_p$.
Put $q_K=p^{f_{K}}$, the order of the residue field of $K$.
By Theorem \ref{MTp}, 
it suffices to show that the characteristic polynomial of 
$D^{K}_{\mrm{cris}}(H^i_{\aet}(X_{\overline{K}},\mbb{Q}_p))$
has integer coefficients and its roots are $q_K$-Weil numbers of weight $i$.
Let $Y$ be the special fiber of 
a proper smooth model of $X$ over the integer ring of $K$.
By the crystalline conjecture shown by Faltings (cf.\ \cite{Fa}), 
we have an isomorphism 
$D^{K}_{\mrm{cris}}(H^i_{\aet}(X_{\overline{K}},\mbb{Q}_p)) 
\simeq K_0\otimes_{W(\mbb{F}_{q_K})} H^i_{\mrm{cris}}(Y/W(\mbb{F}_{q_K}))$
of $\vphi$-modules over $K_0$.  
It follows from Corollary 1.3 of \cite{CLS} 
(cf.\ \cite[Theorem 1]{KM} and \cite[Remark 2.2.4 (4)]{Na})
that the characteristic polynomial of 
$K_0\otimes_{W(\mbb{F}_{q_K})} H^i_{\mrm{cris}}(Y/W(\mbb{F}_{q_K}))$ coincides with 
$\mrm{det}(T-\mrm{Frob}_{\mbb{F}_{q_K}}\mid H^i_{\aet}(X_{\overline{K}},\mbb{Q}_{\ell}))$
for any prime $\ell \not =p$. 
Hence the result follows by the Weil Conjecture (cf.\ \cite{De1}, \cite{De2}).
\end{proof}

Finally, we prove Theorem \ref{MC}. 
Let 
$A$ be an abelian variety over a $p$-adic field $K$ and let $\ell$ be any prime number.
We denote by $T_{\ell}(A)$  the $\ell$-adic Tate module of $A$ 
and set $V_{\ell}(A):=T_{\ell}(A)\otimes_{\mbb{Z}_{\ell}} \mbb{Q}_{\ell}$.
It is well-known that 
we have   $G_K$-equivalent isomorphisms 
$V_{\ell}(A)\simeq H^1_{\aet}(A_{\overline{K}},\mbb{Q}_{\ell})^{\vee}$
and $V_{\ell}(A)/T_{\ell}(A)\simeq A(\overline{K})[\ell^{\infty}]$.
Here, $g$ is the dimension of $A$ and 
 $A(\overline{K})[\ell^{\infty}]$ 
is the $\ell$-power torsion subgroup of $A(\overline{K})$.
Furthermore, for an algebraic extension $L/K$,
the $\ell$-power torsion subgroup $A(L)[\ell^{\infty}]$ of $A(L)$ is finite 
if and only if $V_{\ell}(A)^{G_L}=0$. 
Below we denote by $L$ any finite extension of $Kk_{\pi}$. 
Assume that $A$ has potential good reduction 
and $N_{k/\mbb{Q}_p}(\pi)$ satisfies the condition in the statement of Theorem  \ref{MC}. 
For the proof of Theorem  \ref{MC}, it is enough to show that 
both the $p$-part and 
the prime-to-$p$ part of $A(L)_{\mrm{tor}}$ are finite. \\

\noindent
{\bf Finiteness of the $p$-part of $A(L)_{\mrm{tor}}$ : } 
If we put $W=V_p(A)^{G_L}$, then it is enough to show $W=0$.
Replacing $L$ by a finite extension,
we may suppose that the extension $L/K$ is Galois.
Then the $G_K$-action on $V_p(A)$ preserves  $W$, and thus 
the dual representation $W^{\vee}$ of $W$ is a quotient 
representation of $H^1_{\aet}(A_{\overline{K}},\mbb{Q}_p)$.
By Theorem \ref{MTp:var}, we have $W^{\vee}=(W^{\vee})^{G_L}=0$,
which implies $W=0$ as desired. \\

\noindent
{\bf Finiteness of the prime-to-$p$ part of $A(L)_{\mrm{tor}}$ : } 
The finiteness of the prime-to-$p$ part of $A(L)_{\mrm{tor}}$ 
immediately  follows from the following more general proposition. 

\begin{proposition}
\label{ell-adic}
Let $X$ be a proper smooth variety over  a $p$-adic field $K$
with potential good reduction.
For any prime number $\ell\not=p$,
let  $V$ be a $G_K$-stable subquotient of 
$H^i_{\aet}(X_{\overline{K}},\mbb{Q}_{\ell}(r))$.
Let $L/\mbb{Q}_p$ be an algebraic extension with finite residue field.
Assume $i\not=2r$.

\noindent
{\rm (1)} We have $V^{G_L}=0$ for any $\ell \not=p$.

\noindent
{\rm (2)} Let $T$ be a  $G_K$-stable $\mbb{Z}_{\ell}$-lattice of $V$.
Then we have $(V/T)^{G_L}=0$ for almost all $\ell$. 
\end{proposition}

\begin{proof}
This is essentially shown in Section 4 of \cite{KT}
but we give a proof here for the sake of completeness.
Replacing $K$ and $L$ by finite extensions, 
we may assume that $X$ has good reduction over $K$.
The $G_L$-action on $V$ factors through 
$G_{\mbb{F}_{q_L}}$ where $q_L$ is the order of the residue field of 
$L$.
Let 
$\mrm{Frob}_{L}\in G_{\mbb{F}_{q_L}}$ be the geometric Frobenius.
Put $P_{\ell}(T)=\mrm{det}(T-\mrm{Frob}_L\mid V)$ and 
$Q_{\ell}(T)=\mrm{det}(T-\mrm{Frob}_L\mid H^i_{\aet}(X_{\overline{K}},\mbb{Q}_{\ell}(r)))$.
Clearly $P_{\ell}(T)$ divides $Q_{\ell}(T)$.
By the Weil Conjecture, $Q_{\ell}(T)$ is independent of the choice of $\ell\not=p$,
the coefficients of $Q_{\ell}(T)$  are in $\mbb{Z}[1/q_L]$
and the roots of $P(T)$ are $q_L$-Weil numbers of weight $i-2r$.
Since $i\not= 2r$, we know $Q_{\ell}(1)\not=0$ and then $P_{\ell}(1)\not=0$. 
This in particular shows (1).
To show (2),  it suffices to show
$(T/\ell T)^{G_L}=0$ for almost all $\ell\not =p$. 
Take  any prime number $\ell \not=p$ which is prime to $P_{\ell}(1)$.
For any root $\alpha$ of $P_{\ell}(T)$, we know that $\alpha-1$ is a rational number 
and is an $\ell$-adic unit.
It follows from this fact that 
the action of $\mrm{Frob}_{L}$ on $T/\ell T$
does not have eigenvalue one, which implies $(T/\ell T)^{G_L}=0$. 
\end{proof}

Therefore, we obtained the proof of Theorem  \ref{MC}.

\begin{remark}
\label{normcondi}
(This is pointed out by Yuichiro Taguchi.) 
We can construct an example which gives a negative answer 
to the question given in the Introduction for good reduction case. 
Let $E$ be an elliptic curve over $\mbb{Q}$ with complex multiplication 
by the full ring of integers $\cO_F$ of an imaginary quadratic field $F$.
Let $\psi=\psi_{E/F}$ be the Gr\"ossencharacter
associated with $E$.
Let $p$ be a prime number such that $E$ has good ordinary reduction
and $\mfrak{p}$ a prime ideal of $\cO_F$ above $p$.
If we set $\pi:=\psi(\frak{p})$, then 
$\pi$  is a generator of $\mfrak{p}$  and we have $p=\pi \bar{\pi}$.
Here, $\bar{\pi}$ is the complex conjugation of $\pi$.
Note that $\pi$ is a $p$-Weil number of weight $1$.
Let $K=k$ be the completion of $F$ at $\mfrak{p}$. 
By definition, we have  $K=k=\mbb{Q}_p$
and $\pi$ is a uniformizer of them.
If we identify a decomposition group of $G_F$ at $\mfrak{p}$ with $G_K$,
then the action of $G_K$ on the set of $\pi$-power torsion points of $E(\overline{K})$
is given by the Lubin-Tate character $\chi_{\pi}$ associated with $\pi$.
In particular, we see that $E(Kk_{\pi})[p^{\infty}]$ is infinite.
\end{remark}

\end{document}